\documentclass{amsart}
\usepackage{amsmath}
\usepackage{amssymb}
\usepackage{graphicx}
\usepackage{color}
\usepackage[colorlinks]{hyperref}
\usepackage{graphics,latexsym,tabularx,shapepar}
\usepackage[all,cmtip,2cell,line]{xypic}
\UseAllTwocells\SilentMatrices
 \newtheorem{thm}{Theorem}[section]
 \newtheorem{cor}[thm]{Corollary}

 \newtheorem{lem}[thm]{Lemma}
 
 \theoremstyle{definition}
 \newtheorem{defn}[thm]{Definition}

 \numberwithin{equation}{section}

  \newcommand{\N}{\mathbb{N}}
\newcommand{\Z}{\mathbb{Z}}

 \newcommand{\abs}[1]{\left\vert#1\right\vert}

\begin{document}
\email{dahmadi1387@gmail.com, maliheh.dabbaghian@gmail.com}


\title[]
{Mixing coded systems}

\author[D. Ahmadi, M. Dabbaghian]{Dawoud Ahmadi Dastjerdi, Maliheh Dabbaghian Amiri}

 \begin{abstract}
We show that
a coded system is mixing if and only  if it is totally transitive
 and if in addition  it has a generator whose length of its elements are relatively prime, then it has strong property $P$.
 We continue by showing that a mixing half-synchronized system has such a generator. Moreover, we   give an example of a mixing coded system which does not have any generator whose length of its elements are relatively prime.
\end{abstract}
 \maketitle
\section*{Introduction}

A coded system $X$ is a  subshift over a finite set of alphabet $A$ having a generator set.
 That is, there exists a countable set of words $\mathcal{W}=\{w_i:\; w_i$ is an admissible word$\}$ called \emph{generator} such that any $x\in X\subseteq A^\Z$ is in the closure of the set of arbitrary concatenation of the elements of $\mathcal{W}$.
 For any coded system there is a graph $G$ with countable vertices and  edges that are labeled by the elements of $A$ and so that the closure of recording the labels of all the bi-infinite walks on $G$ determines $X$ \cite{marcus}. The converse is also true, that is if such a graph exists, then that graph represents a coded system.
Coded systems are very chaotic: the periodic points are dense and they are transitive and all have positive entropy.

Subshifts of finite type (SFT) and their factors that is sofic shifts are coded systems and are
among the most investigated subshifts.
Two known properties of sofics which are of our interest for general case are 1) A sofic system is mixing iff it is totally transitive \cite[Theorem 3.3]{bank};  2)
A sofic is mixing iff it has   a generator
$\mathcal{W}$ such that gcd$(\mathcal{W})=\text{gcd}(\{\abs{w_i}:\; w_i\in\mathcal{W}\})=1$.
We will show that (1) is valid for a general coded system but (2) fails to be so.

 The organisation of this paper is as follows. In Section \ref{gen-cod}, we will give some general results for coded subshifts and as one of our main results, we
will prove that  a coded subshift is mixing if and only if it is totally transitive.
Section \ref{Gprime} is devoted to the properties of those coded systems which have a generator
$\mathcal{W}$ such that gcd$(\mathcal{W})=\text{gcd}(\{\abs{w_i}:\; w_i\in\mathcal{W}\})=1$. 
Then we consider a subclass of mixing coded systems with a set of generator $\mathcal{W}$ whose ${\mbox{gcd}}({\mathcal{W}})=1$.
 We show that  this subclass contains the set of half-synchronized systems and all
  have the strong property $P$ in Section \ref{mix-half-syn-sec}.
 So the following implications  are  satisfied for half-synchronized systems.
\begin{eqnarray*}
\mbox{mixing } \Leftrightarrow\mbox{strong property } P &\Leftrightarrow&\\
 K\mbox{-system} \Leftrightarrow \mbox{u.p.e.}&\Leftrightarrow& \mbox{weak mixing} \Leftrightarrow \mbox{totally transitive}.
\end{eqnarray*}
In Section \ref{subsec-cod}, we will construct an example of a mixing coded system whose any generator $\mathcal{W}$ has gcd$(\mathcal{W})>1$.

\section{ Definitions and preliminaries}
A TDS is a pair $( X,\,T)$
such that $X$ is a compact metric space and $T$ is a homeomorphism.
 \emph{The return time} set is defined to be
$$N(U,\,V)=\{n\in\mathbb{Z}: U\cap T^{-n}V\neq\emptyset\}$$   where $U$ and $V$ are
opene  (nonempty  and  open)  sets.  A  TDS  $(X,\,T)$   is
\emph{transitive} if  $N(U,\,V)\neq\emptyset$; and it is
\emph{totally transitive} if $(X,\,T^n)$ is transitive for any $n\in\Z$.
We call a TDS   $(X,\,T)$
\emph{weak mixing}
 if $N(U,\,V)$ is a thick set  (i.e. containing
arbitrarily long intervals of $\Z$) for any two opene sets $U$ and  $V$;
 and
is \emph{strong mixing} if  $N(U,\,V)$ is cofinite for opene sets $U,\,V$.


Let $A$ be a finite alphabet, i.e. a finite set of symbols. The shift map $\sigma:A^{\Z}\to A^{\Z}$
 is defined by $\sigma((a_i)_{i\in\Z}) = (a_{i+1})_{i\in\Z}$, for $(a_i)_{i\in\Z}\in A^{\Z}$. If $A^{\Z}$ is endowed
with the product topology of the discrete topology on $A$, then $\sigma$ is a homeomorphism and $(A^{\Z},\,\sigma)$ is a TDS called \emph{two-sided shift} space. Similarly, \emph{one-sided shift} space can be defined on $A^{\N_0}$,
 then $\sigma$ is  a finite-to-one continuous map.
 A \emph{subshift} is the restriction of $\sigma$ to any closed
non-empty subset $\Sigma$ of $A^{\Z}$ that is invariant under $\sigma$. Set $\sigma_\Sigma:=\sigma\vert_\Sigma$.
 A \emph{word} (\emph{block}) of length $n$ is $a_0 a_1\cdots a_{n-1}\in A^n$ if there is
 $x\in \Sigma$ such that $x_i=a_i,\, 0\leq i \leq n-1$.
 The \emph{language} $\mathcal{L}(\Sigma)$ is the  collection of  all words of $\Sigma$ and $\mathcal{L}_n(\Sigma)$ is the collection of all words in $\Sigma$ of length $n$.
   Also, a \emph{cylinder} is defined as $[a_0\cdots a_n]_p^q=\{x\in\Sigma: x_p=a_0,\ldots,\,x_q=a_n\}$.

Shift  spaces described by a finite set of forbidden  blocks are called \emph{shifts
of finite type} (SFT) and their factors are called \emph{sofic}.
A word $w\in\mathcal{L}(\Sigma)$ is called \emph{synchronizing} if  whenever $uw,\,wv\in\mathcal{L}(\Sigma)$, then $uwv\in\mathcal{L}(\Sigma)$. A \emph{synchronized system} is an irreducible shift which has a synchronizing word. 
{
A subshift $\Sigma$ is specified, or has specification property, if there is $N\in\N$ such that if $u,\; v\in\mathcal{L}(\Sigma)$, then there is $w$ of length $N$ so that 
$uwv\in\mathcal{L}(\Sigma)$. A specified system is mixing and synchronized and any mixing sofic is specified.
}

\section{General Results in Coded Systems}\label{gen-cod}
The coded systems were first defined by Blanchard and Hansel in \cite{blaH} as an extension for sofic systems.  There are several equivalent definitions and we choose the following.
\begin{defn}
A shift space is \emph{coded} if its language $\mathcal{L}$ is freely generated by concatenating  words in a countable set $\mathcal W$  called the \emph{generator} set.
We denote a coded space $X$  by $X(\mathcal{W})$.
\end{defn}

A well studied sub-class containing  sofics are  synchronized systems.
The following result was already known for sofics \cite[Theorem 3.3]{bank} and then extended for synchronized systems \cite[Proposition 4.8]{MO}.
\begin{thm}\label{coded-eq}
A coded system is mixing if and only if it is totally transitive.
\end{thm}
\begin{proof}
The necessity is clear. So let $X$, our coding system, be totally transitive. Since a coded system has a dense set of periodic points, $X$ is weak mixing \cite[Corollary 3.6]{huang2}.

{
Let $\mathcal{W}$ be the generator of  $X$ and
 $\gcd(\mathcal{W})=k$. First assume that there is $u_0\in\mathcal{W}$ such that $|u_0|=k$.

 To see that $X$  is mixing, it suffices to show that for any $u,\,v\in\mathcal{W},\,N([u],\,[v])$ is cofinite.
 We do this by showing that there exists some $m\in\mathbb{N}$ such that $N([u],\,[v])-|u|\supset k\mathbb{N}\cup
 (k\mathbb{N}+mk+1) \cup\cdots\cup (k\mathbb{N}+mk+k-1)$. As a result, $N([u],\,[v])-|u|$ will contain all numbers greater than $(m+1)k$ and so it is cofinite.

Since $u_0,\,u,\,v\in\mathcal{W}$, for any $n$ we have $u(u_0)^nv\in\mathcal{L}(X)$ and this implies that $k\mathbb{N}\subset N([u],\,[v])-|u|$.
But $X$ is weak mixing and  $N([v],\,[u])$ must be thick. Hence there does exist some $s\in\mathbb{N}$ such that $s-1,\,s-2,\ldots,\,s-k\in N([v],\,[u])-|v|$. This means that  $u\alpha_i v\in\mathcal{L}(X)$
where $|\alpha_i|=s-i$.

Observe that for any $i$, $u\alpha_i v$ is a subword of
\begin{equation}\label{w_i}
 w_i=u_iu\alpha_i vv_i
\end{equation}
where $w_i$ is constructed by
some concatenation of some words in
$\mathcal{W}$.  By the fact that $\gcd(\mathcal{W})=k$,  $|w_i|,\,|u|,\,|v|\in k\mathbb{N}$ and
this implies that $|u_i|+|v_i|+s-i\in k\mathbb{N}$. Without loss of generality  assume that $s\in
k\mathbb{N}$. Hence $|u_i|+|v_i|\equiv i (\mod k)$ and so there is $m_i\in\mathbb{N}$ such that $|u_i|+|v_i|=m_ik+i$.
 By considering \eqref{w_i},  for any $n\in\mathbb{N}$,   $vv_i(u_0)^nu_iu\in\mathcal{L}(X)$
and this in turn implies that $nk+m_ik+i\in N([u],\,[v])-|u|$.

 Since $n$ is arbitrary, by setting  $m=\max\{m_1,\ldots,\,m_k\}$, for any $i$ we have $k\mathbb{N}+mk+i\in N([u],\,[v])-|u|$.

Now suppose for some $l>1$, $|u_0|=lk$. Replace $u$ and $v$ with $u^l$ and $v^l$ respectively and we consider $w_i^l$ instead of $w_i$. So $|u|,\,|v|,\,|w_i|$ are in $lk\mathbb{N}$ and  by a similar reasoning as above we are done.
}
\end{proof}
{

\section{Mixing coded systems with a relatively prime generator
}\label{Gprime}
A generator $\mathcal{W}$ is called \emph{relatively prime} if gcd$\{|w_i|:\; w_i\in\mathcal{W}\}=1$.
Note  that if a coded system has a relatively prime generator, then it possesses a generator $\mathcal{W}'$ having two elements with coprime lengths.
For there is sufficiently large $k\in \N$ such that  gcd($\{|w_1|,\ldots, |w_k|\})=1$.
By B\'ezout's identity and a reindexing of  $\{w_1,\ldots, w_k\}$ if necessary, there are $y_i\in\N$ with $1=y_1|w_1|+\cdots+y_\ell |w_\ell|-(y_{\ell+1}|w_{\ell+1}|+\cdots+y_k|w_k|)$.
Now, gcd($y_1|w_1|+\cdots+y_\ell |w_\ell|,\, y_{\ell+1}|w_{\ell+1}|+\cdots+y_k|w_k|)=1$
and so we may set $$\mathcal{W}'=\mathcal{W}\cup\{w_1^{y_1}\cdots  w_\ell^{y_\ell},\, w_{\ell+1}^{y_{\ell+1}}\cdots w_k^{y_k}\}.$$
Henceforth, without loss of generality, we may assume that any relatively prime generator has two elements with coprime lengths.

 An easy implication of the following lemma is that a coded system with a relatively prime generator is mixing.
\begin{lem}\label{coprime}
Suppose  $\gcd\{a_1,\,a_2\}=1$,  $a_1,\, a_2\in\N$.
 Then there is $L\in\N$ such that for any $n\geq L$, there are $r_1,\, r_2\in\N_0=\mathbb{N}\cup\{0\}$ with $n=r_1a_1+r_2a_2$.
\end{lem}

Later in Theorem \ref{mix-cod-gcd},
we will show that there are examples of mixing coded systems which do not have any relatively prime generator.  However, a trivial argument shows that all mixing synchronized do have a relatively prime generator.
{
 In order to see this,  consider
\begin{equation}\label{gen-syn}
\mathcal{W}=\{w\alpha: \ \alpha w\alpha\in\mathcal{L}(X), \alpha\not\subset w\}
\end{equation}
as a generator for $X$
where $\alpha$ is a synchronized word;
then the assertion follows from the fact that $N([\alpha],\,[\alpha])$ is cofinite.

A TDS $(X,\/ T)$ is said to have \emph{strong property} $P$ \cite[Definition 6.2]{huang1}, if for any finite
opene sets $U_1,\,U_2,\, \ldots ,\,U_n$ in $X$ there exists $N\in \N$ such that for any $k\geq 2
$ and any $s = (s(1), s(2), \cdots , s(k))\in  \{1,\,2,\ldots ,\,n\}^k$ there exists $x \in X$ with
\begin{equation}\label{eq:sp}
 x\in U_{s(1)} \cap T^{-N}U_{s(2)}\cap \cdots \cap T^{-(k-1)N}U_{s(k)}.
\end{equation}
Clearly a TDS with specification property has strong property $P$. The converse is not true though. In fact, Blanchard in \cite[Example 5]{bl}   gives an example
 of a subshift which has u.p.e but  not mixing; so it cannot have specification property. By \cite[Proposition 4]{bl}, Blanchard's example has strong property $P$.
A consequence of this argument is that mixing in a general TDS
is different from having property P and does not imply this property.
However, a TDS with property $P$ is always weak mixing \cite{bl}. The situation is different
 for a coded system with a relatively prime generator.

\begin{thm}
\label{thm:sp}
A  coded system with a relatively prime generator  has strong property $P$.
\end{thm}
\begin{proof}
Let $X$ be mixing coded and $\mathcal{W}$ a generator for $X$ with $\mbox{gcd}(\mathcal{W})=1$.
First assume that  $v^{(1)}$ and $v^{(2)}$ are two words in
$\mathcal{W}$ with
   coprime lengths.
Since any concatenation of $v^{(1)}$ and $v^{(2)}$ is again in $\mathcal{L}(X)$,  by Lemma \ref{coprime},  there is $L\in\N$ such  that for  $m\geq L$
 there are some concatenations of $v^{(1)}$ and $v^{(2)}$ of length $m$.

To show that $X$ has strong property $P$, let $\mathcal{U}=(U_1,\ldots,\,U_n)$  be an $n$-tuple of nonempty
 open sets and pick  $u_i\in\mathcal{L}(X)$ such that $[u_i]\subset U_i$.
  Moreover, as a coded system there is $V^{(i)}$ a concatenation of some words in $\mathcal{W}$ such that  $u_{i}\subseteq V^{(i)}$ and let $V^{(i)}=v_{i}u_{i}w_{i}$.
We may further assume that there are $l$ and $N$ such that for all $i$,
(1) $l=|v_i|$ and (2) $N=|V^{(i)}|$. This is possible for otherwise
set $l'=L+N_0$ where $N_0=\max\{|v_i|:\  1\leq i\leq k\}$
and
 extend each $V^{(i)}$  from left by concatenation of some $v ^{(1)}$ and $v^{(2)}$ whose sum of lengths is
 $l'-|v_i|$ and call it $V^{(i)}$ again. This sets (1) and similar reasoning by adding necessary concatenation of $v^{(1)}$ and $v^{(2)}$ from right  achieves (2).
Observe that by this construction, the concatenation of any order of $V^{(i)}$'s is a word in $X$ and since
$u_i$ starts from
 entry $(l+1)$ of $V^{(i)}$, we have
$$[u_{s(1)}] \cap \sigma^{-N}[u_{s(2)}]\cap \cdots \cap \sigma^{-(k-1)N}[u_{s(k)}]\neq \emptyset$$
where $s = (s(1),\, s(2), \ldots , \,s(k))\in  \{1,\,2,\ldots ,\,n\}^k$.
By allowing $T=\sigma$,
this implies that there exists $x \in X$ such that \eqref{eq:sp}
is satisfied.
\end{proof}
By \cite[Lemma 6.3]{huang1}, any system with strong property $P$ is $K$-system. In fact, we have the following implications
$$\mbox{strong property } P \Rightarrow K\mbox{-system} \Rightarrow \mbox{u.p.e.}\Rightarrow \mbox{weak mixing} \Rightarrow \mbox{totally transitive}.$$
By the above results, all the converses are true for the coded systems with a relatively prime generator.

\subsection{Mixing half-synchronized systems}\label{mix-half-syn-sec}
As we saw all mixing synchronized systems have a relatively prime generator. We extend this result to half-synchronized systems. First let us recall the necessary definitions.

For any $x\in X$, $x_-=x_{(-\infty,\,0]}$ (resp. $x_+=x_{(0,\,\infty)}$) is called a \emph{left ray} (resp. \emph{right ray}).
 Set $X^-=\{x_-: x\in X\}$ and $X^+=\{x_+: x\in X\}$.
It is abvious that for any $u\in\mathcal{L}(X)$  and any left (resp. right) ray $x_-\in X^-$ (resp. 
$x_+\in X^+$), $x_-u$ (resp. $ux_+$) is  a  left (resp. right) ray.

Let $\eta$ be either a word or a left ray. The \emph{follower set} of $\eta$ is defined as
$\omega_+(\eta)=\{x_+\in X^+:\; \eta x^+ \  \mbox{is admissible}\}$. The \emph{follower block} of $\eta$  is denoted by $F_+(\eta)=\{u\in \mathcal{L}(X):\; \eta u \ \mbox{is admissible}\}$.
Similarly one may define the \emph{precessor set} $\omega_-(\zeta)$ or \emph{precessor block}  $F_-(\zeta)$ of a word or a right ray $\zeta$.

\begin{defn}\label{half-syn-defn}\cite[Definition 0.9]{fiebig}
A transitive subshift $X$ is \emph{half-synchronized} if  there is  $m\in\mathcal{L}(X)$,
 called the \emph{half-synchronizing} word of $X$, and a left transitive ray $x_-\in X$ such that $x_{-|m|+1}\cdots x_0=m$ and $\omega_+(x_-)=\omega_+(m)$.
\end{defn}
One can change the condition $F_+(x_-)=F_+(m)$ with
$\omega_+(x_-)=\omega_+(m)$ in Definition \ref{half-syn-defn}.
Also, similar definition for a right transitive ray can be given.
\begin{thm}\label{half-syn-mix}
A half-synchronized system is mixing if and only if it has a relatively prime generator.
\end{thm}
\begin{proof}
The sufficiency is obviously true. To prove the necessity,  we first define a generator $\mathcal{W}$ for a given
half-synchronizing system and  we will show that if the system is   mixing, then  $\mathcal{W}$ is  relatively prime.
  Let $(X,\,\sigma)$ be a half-synchronized system and let $x$ and $m$ be  as provided in Definition \ref{half-syn-defn}.
 Define
\begin{equation}\label{gensyn}
\mathcal{W}=\{wm: \ mwm=x_{-n}\cdots x_0\subset x_{(-\infty,\,0]} \mbox{ for some } n\}.
\end{equation}

 Since $x$ is left transitive, all words of $\mathcal{L}(X)$ appear infinitely many times as asubword in $x_-$, in other words, $\mathcal{L}(X)\subseteq \mathcal{L}(\Sigma(\mathcal{W}))$.

We prove our assertion by showing that  $\mathcal{L}(\Sigma(\mathcal{W}))\subseteq \mathcal{L}(X)$.
This is done by an induction argument for showing that all concatenations of members of $\mathcal{W}$ are in $\mathcal{L}(X)$.
It is trivially true for
 any two concatenations.
So assume that this is true for $k-1$ concatenations and consider a $k$ concatenations of members of $\mathcal{W}$ such as
 $w_{i_1}mw_{i_2}m\cdots w_{i_k}m$. By assumption $mw_{i_2}m\cdots w_{i_k}m$ is admissible; so $w_{i_2}m\cdots w_{i_k}m\in\omega_+(m)=\omega_+(x_-)$.
But $w_{i_1}$ appears as a subword in the end of $x_-$, thus we have $w_{i_1}mw_{i_2}m\cdots w_{i_k}m\in\mathcal{L}(X)$ and so $\mathcal{L}(\Sigma(\mathcal{W}))\subseteq \mathcal{L}(X)$ as required.

Now suppose that $X$ is mixing
and assume $\gcd(\mathcal{W})=k\geq 2$. So for $wm\in\mathcal{W}$,
\begin{equation}\label{wmkN}
 |wm|\in k\N.
\end{equation}
 But
  $N([m],\,[m])$ is cofinite; so
 for a fixed $1\leq i\leq k-1$, there is $u_i\in\mathcal{L}(X)$ such that $mu_im\in\mathcal{L}(X)$ and $|u_im|\in k\mathbb{N}+i$.
 By left transitivity of $x$, there is  some $t_i\in\mathbb{N}$ such that $mu_im=x_{-t_i}\cdots x_{-t_i+|mu_im|}$.
 On the other hand,   definition of $\mathcal{W}$ implies that there are some $v_i\in\mathcal{L}(X)$ so that $v_im,\,u_imv_im\in\mathcal{W}$
 and from (\ref{wmkN}), both $|v_im|$ and
 so $|u_imv_im|$ must be in $k\mathbb{N}$. However, $|u_imv_im|=|u_im|+|v_im|\in k\mathbb{N}+i$
and so our assumption is false and $k=1$ as required.
\end{proof}

\begin{cor}
A half-synchronized system is mixing if and only if it has strong property $P$.
\end{cor}
\begin{proof}
A system with strong property $P$ is weakly mixing and thus totally
transitive
and in fact mixing  by Theorem \ref{coded-eq}.
The converse follows from Theorem \ref{half-syn-mix} and Theorem \ref{thm:sp}.
\end{proof}

The generator given in (\ref{gen-syn}) is a handy one for synchronized systems. For a half-synchronized system we have:
\begin{thm}\label{con-gen}
Fix a word $m$ and let $\mathcal{U}$ be the collection of all words
 such that  whenever $u=u_1mu_2\in \mathcal{U}$, we have $u_2\in \mathcal{U}$. Set $\mathcal{W}=\{um: u\in\mathcal{U}\}$; then
 $X=X(\mathcal{W})$ is  half-synchronized with a half-synchronizing word $m$.
\end{thm}
\begin{proof}
Choose a left transitive point  $x_-$ which is a
 concatenation of words in $\mathcal{W}$. Thus $x_{[-|m|+1\cdots 0]}=m$ and we have
 $\omega_+(x_-)\subset\omega_+(m)$. On the other hand, since any concatenation of $\mathcal{W}$ follows $x_-$, $\omega_+(m)\subset\omega_+(x_-)$ and consequently $\omega_+(m)=\omega_+(x_-)$.
\end{proof}

The converse to the above theorem is also correct, that is,
 if $m$ is a half-synchronizing word, then there is $\mathcal{U}$ satisfying the hypothesis of the above theorem.
 For let $X$ be half-synchronized and $\mathcal{W}$ its generator as (\ref{gensyn}).
 Now this generator is the same as the one provided in Theorem \ref{con-gen} by letting $\mathcal{U}=\{w: wm\in\mathcal{W}\}$.

Note that  we are not assuming that whenever $u=u_1mu_2\in \mathcal{U}$, then $u_1\in\mathcal{U}$.
 Hence, using Theorem \ref{con-gen}, one can construct examples of \emph{strictly half-synchronized} system, that is systems which are half-synchronized but not synchronized.

\section{Coded mixing systems without relatively prime generator}\label{subsec-cod}

We give a collection of examples of coded mixing  systems without any reatively prime generator.
\begin{thm}\label{mix-cod-gcd}
For any $k\in\N$ there is a coded mixing  system such that if $\mathcal{W}$ is any generator, then gcd$(\mathcal{W})\geq k\geq 2$.
\end{thm}
\begin{proof}
Assume that $k=2$ and choose $X$ to be a non-trivial two sided mixing minimal subshift in $\{0,\,1\}^\Z$ and let $x=\{x_b\}_{b\in\Z}\in X$.
 Let $S\subset \{0,\, 1,\,2,\,3\}^\Z$ be the coded system
given by the cover $\mathcal{G}$ in
Figure \ref{fig1} where $y_b=x_b+2$. 
Denote by  $\pi_u$ a path 
labeled $u=u_0u_1\cdots u_{\abs{u}-1}$
in $\mathcal{G}$
  and let $m_u$ (resp.  $M_u$) be the integer assigned to  the most left (resp. right) vertex appearing in $\mathcal{G}$ for $\pi_u$ including the initial and terminal vertices. 
For instance, $\pi_{x_0x_1y_1}$ in Figure \ref{fig1} has $m_{x_0x_1y_1}=0$ and $M_{x_0x_1y_1}=2$.
Set $\pi_{x(u)}$ to be the path labeled
 $x(u)=x_{m_u}x_{m_u+1}\cdots x_{(M_u-1)}\in\mathcal{L}(X)$ initiating at $m_u$.
Observe that $\{y_b\}_{b\in\Z}$ with obvious correction on indices defines a minimal subshift $Y\subset \{2,\,3\}^\Z$ conjugate to $\sigma_X^{-1}$. Let $y(u)=y_{(M_u-1)}y_{(M_u-2)}\cdots y_{m_u}\in\mathcal{L}(Y)$
and note that $x(u)$ and $y(u)$ are not necessarily subwords of $u$, but if $\pi_u$ and $\pi_{u'}$ are different paths both labeled by $u$, then
\begin{equation}\label{mosavi}
x(u)=x(u')\qquad \text{and} \qquad y(u)=y(u').
\end{equation}
\begin{figure}[htbp]
\centering
$\cdots \xymatrix{{\SMALL-2} &{\SMALL-1} \ltwocell_{x_{-2}}^{y_{-2}}{'} & {\SMALL 0}\ltwocell_{x_{-1}}^{y_{-1}}{'} & 
{\SMALL 1}\ltwocell_{x_{0}}^{y_{0}}{'} &{\SMALL 2} \ltwocell_{x_{1}}^{y_{1}}{'} & {\SMALL 3}\ltwocell_{x_{2}}^{y_{2}}{'} & {\SMALL 4}\ltwocell_{x_{3}}^{y_{3}}{'}}\cdots $
\caption{Cover $\mathcal{G}$ for $S$.}\label{fig1}
\end{figure}
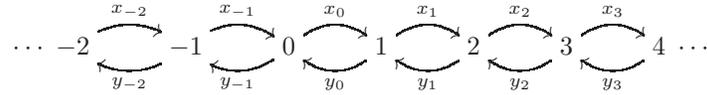
First, we show that $S$ is mixing.
Let $u,\,u'\in\mathcal{L}(S)$ and by necessary prolongations, we may
assume that if $\pi_u$ (resp. $\pi_{u'}$) is a path labeled $u$ (resp. $u'$), then $\pi_u$ (resp. $\pi_{u'}$) starts at $m_u$ (resp. $m_{u'}$) and terminates at $M_u$ (resp. $M_{u'}$), or in fact, $t(\pi_u)=\tau_{x_{(M_u)-1}}$
 (resp. $i(\pi_{u'})=\imath_{u'_m}$).
Since $X$ is mixing, there exists $K$ and paths labeled $w_i\in\mathcal{L}(X)$, $\abs{w_i}\in\{K,\,K+1,\ldots\}$ in $\mathcal{G}$ so that $\pi_{uw_iu'}$ is again a path in $\mathcal{G}$.
This means  $uw_iu'\in \mathcal{L}(S)$ or equivalently $S$ is mixing.

Now we show that if $\mathcal{W}=\{w_0,\,w_1,\ldots\}$ is any generator for $S$, then gcd$(\mathcal{W})\geq 2$.
Since $w_i^\infty\in S$, we do this by showing that if $u^\infty$ is any periodic point in $S$, then $\pi_u$ is a cycle in $\mathcal{G}$ and so $\abs{u}$ as well as $\abs{w_i}$
 are  even numbers. In particular, by assuming $u$ being in its least period $\left((u_1\subseteq u, \;
  u_1^\infty=u^\infty) \Rightarrow u_1=u\right)$, then
 $\abs{(x(u))^\ell}=\abs{x(u)}$ is constant for all $\ell$.

 If $u^\infty\in S$  is not the image of any closed cycle in $\mathcal{G}$, then for some sufficiently large $\ell_0$ and for all $\ell\geq \ell_0$, there is a path $\pi_{u^\ell}$
 in $\mathcal{G}$ with $\sup\{\abs{x(u^\ell)}:\; \ell\geq\ell_0\}=\infty$.
By choosing an appropriate subsequence, we may assume that $\lim_\ell \abs{x({u^\ell})}\nearrow\infty$.

Without loss of generality assume that  $\pi_u$ starts at its respective
 $m_u$ and $u$ is in its least period. 
By these assumptions, $x({uu})=x(u)x(u)$ and so $x(u^\ell)$  is a word in $X$ for $\ell\geq\ell_0$.  
This in turn means that $X$ has to have a periodic point $x(u)^\infty$ which  is absurd for a minimal system such as $X$.

Now for an arbitrary $k\geq 2$, let $X$ be as above, but  $\{x_i\}_{i\in\Z}$  the orbit of a point $x\in X$ under $\sigma_X^{k-1}$.
\end{proof}

Note that any half-synchronized subsystem of $S$ in the above theorem is not mixing. In particular, one can choose non-mixing increasing SFT's $X_1\subset X_2\subset\cdots$ such that $X=\overline{\cup_{i\in\N}X_i}$. That establishes an example of non-mixing SFT's converging to a mixing coded system.  
{
See also \cite[Example 4.11]{MO}, for non-mixing SFT's approximating a two sided full shift.
}

Also,  it is not hard to see that the examples of Theorem \ref{mix-cod-gcd} all have strong property P. So the converse of Theorem \ref{thm:sp} is not necessary true. 


\end{document}